\documentclass[11pt,oneside,pagesize=auto]{scrartcl}

\newcommand{\sep}{,}

\newcommand{\titl}{Some New Bounds on the Entropy Numbers of Diagonal Operators}

\newcommand{\keyw}{Diagonal Operators\sep{} Entropy Numbers}

\usepackage{tocbasic}
\usepackage{scrlayer-scrpage}
\usepackage[onehalfspacing]{setspace}

\KOMAoptions{mpinclude=false}
\KOMAoptions{headsepline=false}
\KOMAoptions{footsepline=false}

\KOMAoptions{titlepage=false}

\KOMAoptions{bibliography=totoc}

\KOMAoptions{draft=false} 

\KOMAoptions{abstract=true}

\KOMAoptions{BCOR=0cm, DIV=12}

\addtokomafont{title}{\boldmath}
\addtokomafont{section}{\boldmath}
\addtokomafont{subsection}{\boldmath}
\addtokomafont{paragraph}{\boldmath}
\addtokomafont{sectionentry}{\boldmath}

\usepackage[latin1]{inputenc}  
\usepackage[english]{babel}
\usepackage{microtype}
\usepackage{enumitem}

\usepackage{amsmath}                 
\usepackage{amssymb}
\usepackage{amsthm}
\usepackage{mathrsfs}
\usepackage{mathtools}
\usepackage{bbm}

\usepackage[numbers,square]{natbib}

\usepackage{multicol}
\usepackage{multirow}

\usepackage{tikz}

\usepackage[
pdfusetitle, 
pdfkeywords={\keyw}, 
backref = false, 
colorlinks = true, 
pdfpagelabels = true, 
pdfstartview = FitH, 
bookmarksopen = true, 
bookmarksnumbered = true, 
linkcolor = black, 
plainpages = false, 
hypertexnames = false, 
citecolor = black 
] {hyperref}

\usepackage{comment}

\setlist{itemsep=0pt}
\setenumerate[1]{label={(\roman*)}}
\setenumerate[2]{label={(\alph*)}}

\newtheoremstyle{famous} 
{} 
{} 
{\itshape}
{} 
{\bfseries\sffamily\boldmath} 
{} 
{\newline} 
{\boldmath\thmnumber{\begin{footnotesize}#2\end{footnotesize} }\thmnote{#3}}

\newtheoremstyle{kursiv}
{} 
{} 
{\itshape} 
{} 
{\bfseries\sffamily\boldmath} 
{} 
{ } 
{\boldmath\thmnumber{\begin{footnotesize}#2\end{footnotesize} }\thmname{#1 }\thmnote{(#3)}}
  
\newtheoremstyle{normal}
{} 
{} 
{\rmfamily} 
{} 
{\bfseries\sffamily\boldmath} 
{} 
{ } 
{\boldmath\thmnumber{\begin{footnotesize}#2\end{footnotesize} }\thmname{#1 }\thmnote{(#3)}}
  
\newtheoremstyle{oNum}
{} 
{} 
{\rmfamily} 
{} 
{\bfseries\sffamily\boldmath} 
{} 
{\newline} 
{\thmname{#1 }\thmnote{(#3)}}
  
\makeatletter
\def\@endtheorem{\endtrivlist}
\makeatother
   
\theoremstyle{kursiv}
\newtheorem{thm}{Theorem}[section]

\newtheorem{lem}[thm]{Lemma}

\theoremstyle{normal}

\theoremstyle{famous}

\ifdefined\verboseproofs

	\excludecomment{shortproof}

\else
	\excludecomment{longproof}

\fi

\DeclareFontFamily{U}{matha}{\hyphenchar\font45}
\DeclareFontShape{U}{matha}{m}{n}{
      <5> <6> <7> <8> <9> <10> gen * matha
      <10.95> matha10 <12> <14.4> <17.28> <20.74> <24.88> matha12
      }{}
\DeclareSymbolFont{matha}{U}{matha}{m}{n}
\DeclareFontFamily{U}{mathx}{\hyphenchar\font45}
\DeclareFontShape{U}{mathx}{m}{n}{
      <5> <6> <7> <8> <9> <10>
      <10.95> <12> <14.4> <17.28> <20.74> <24.88>
      mathx10
      }{}
\DeclareSymbolFont{mathx}{U}{mathx}{m}{n}
\DeclareMathDelimiter{\vvvert}{0}{matha}{"7E}{mathx}{"17}

\DeclareNewTOC[%
type=todo,%
types=todos,%
name=Todo,%
listname={Todo Liste}%
]{lotodo}
\setuptoc{lotodo}{chapteratlist} 

\newcounter{todono}
\renewcommand{\thetodono}{(\arabic{todono})}

\newcommand{\todo}[2]{\stepcounter{todono}
$^{\text{\thetodono}}$\marginline{
\begin{singlespace}
\footnotesize
\textbf{{\tiny\thetodono} #1}\par #2
\end{singlespace}}%
\addcontentsline{lotodo}{table}{#1}}

\newcounter{temp}

\renewcommand{\d}{{\mathrm d}}

\newcommand{\R}{{\mathbb R}}

\newcommand{\N}{{\mathbb N}}

\DeclareMathOperator{\id}{id}

\newcommand{\uball}[1]{B_{#1}} 

\newcommand{\eqspace}{\;}

\newcommand{\sfrac}[2]{{#1}/{#2}}

\newcommand{\entropy}[2]{\varepsilon_{#1}(#2)}

\newcommand{\covering}[3][]{
\def\tempsize{#1}
\ifx\tempsize\empty
	\mathcal{N}(#3,#2)
\else
	\mathcal{N}\csname #1l\endcsname{(}#3,#2\csname #1r\endcsname{)}
\fi
}
\newcommand{\packing}[2]{\mathcal{P}(#2,#1)}

\newcommand{\quasi}[1]{\kappa_{#1}}

\newcommand{\ALP}{\hyperref[it:alp]{ALP}}
\newcommand{\AMP}{\hyperref[it:amp]{AMP}}
\newcommand{\EXP}{\ref{eq:exp}}

\renewcommand{\todo}[2]{} 

\title{\titl}
\author{Simon Fischer}
\publishers{\normalsize Institute for Stochastics and Applications\\
Faculty 8: Mathematics and Physics\\
University of Stuttgart\\
D-70569 Stuttgart Germany \\
\texttt{\small simon.fischer@mathematik.uni-stuttgart.de}}
\date{\normalsize\today}

\begin{document}


\maketitle

\begin{abstract}
Entropy numbers are an important tool for quantifying the compactness of operators. Besides establishing new upper bounds on the entropy numbers of diagonal operators $D_\sigma$ from $\ell_p$ to $\ell_q$, where $p\not=q$, we investigate the optimality of these bounds. In the case of $p<q$ optimality is proven for fast decaying diagonal sequences, which include exponentially decreasing sequences. In the case of $p>q$ we show optimality under weaker assumption than previously used in the literature. In addition, we illustrate the benefit of our results with examples not covered in the literature so far.
\end{abstract}





\section{Introduction and Main Results}


Entropy numbers and covering numbers are important standard tools for quantifying the compactness of operators with various applications in different fields of mathematics, e.g.\ 
functional analysis (see e.g.\ \cite{K1986,CaSt1990,EdTr1996} for operator ideals and eigenvalue distribution of compact operators), 
approximation theory (see e.g.\ \cite{Tr1978,EdTr1996,Tr2006} for embeddings of Sobolev or Besov spaces), 
probability theory (see e.g.\ \cite{KuLi1993,LiLi1999} for small deviations of Gaussian processes and \cite{VaWe1996} for empirical process theory), and
statistical learning theory (see e.g.\ \cite{ScSm2001,GyKoKrWa2002,CuZh2007,StCh2008} for capacity of hypothesis spaces).
In many of these applications discretization techniques are used to reduce the often difficult problem of estimating entropy numbers in function spaces to easier estimation problems in sequence spaces. For instance, the problem of quantifying the compactness of Sobolev embeddings can be reduced to diagonal operators in sequence spaces via wavelet or Fourier bases, see e.g.\ \cite{K2008,CoK2009} and references therein. 
%
%
In this article, we therefore derive new entropy number bounds for diagonal operators.

To be more precise, let $0< p,q\leq \infty$ and $\sigma=(\sigma_k)_{k\geq 1}$ be a non-negative and non-increasing sequence of real numbers. We write $D_\sigma:\ell_p\to\ell_q$ for the diagonal operator between the usual sequence spaces $\ell_p$ and $\ell_q$, i.e.\ $D_\sigma(x_k)_{k\geq 1} \coloneqq (\sigma_kx_k)_{k\geq 1}$. If we denote the closed unit ball of $\ell_p$ by $\uball{\ell_p}$ then the entropy numbers of the operator $D_\sigma:\ell_p\to\ell_q$ are defined by
\[ 
\entropy{n}{D_\sigma} \coloneqq \inf\Bigl\{\varepsilon>0:\ \exists y_1,\ldots,y_n\in \ell_q\text{ with }D_\sigma\uball{\ell_p}\subseteq\bigcup_{i=1}^ny_i + \varepsilon\uball{\ell_q}\Bigr\}
\]
for all $n\geq 1$. In case of $p=q$ the asymptotic behavior of the entropy numbers $\entropy{n}{D_\sigma}$ is well-known for \emph{all} diagonal sequences $\sigma$, see e.g.\ \citet[Proposition~1.7]{GoKSc1987} for the Banach space case $1\leq p \leq \infty$ but, modulo the constant, the result remains valid for all $0<p\leq\infty$. In case of $p\not=q$---as far as we know---there are only partial answers, see e.g.\ \cite{K2005,K2008,CaRu2014}. The present work is a further contribution to this problem: 
Our first theorem fills a gap in the literature by providing an upper bound in case of $p<q$, which is optimal for sequences 
satisfying the condition \emph{exponential decay} (\EXP), see Theorem~\ref{thm:p_leq_q} for an exact definition. The second theorem considers the case $p>q$ and gives an upper bound, which is optimal for sequences satisfying the condition \emph{at least polynomial decay} (\ALP) as well as for sequences satisfying the condition \emph{at most polynomial decay} (\AMP), see Theorem~\ref{thm:p_geq_q} for an exact definition of (\ALP) and (\AMP).
For the  second type of sequences this recovers the optimal bound of \citet{K2008}, while the first type of sequences have not been considered so far. A more detailed comparison between our results and existing bounds can be found at the end of this section. The proofs of both our theorems combine the ideas of \citet[Proposition~1.7]{GoKSc1987} and \citet[Hilfsatz~2]{Ol1978}. Moreover, in the appendix we summarize relations between the regularity conditions on $\sigma$ we consider and some other common regularity conditions.

Before we proceed let us introduce some notation. For real sequences $(x_n)_{n\geq 1}$ and $(y_n)_{n\geq 1}$ we write $x_n\preccurlyeq y_n$ iff there is a constant $c>0$ with $x_n\leq c y_n$ for all $n\geq 1$ and $x_n\asymp y_n$ iff $x_n\preccurlyeq y_n$ as well as $x_n\succcurlyeq y_n$ hold. In the following, we declare an upper or lower bound $(x_n)_{n\geq 1}$ on the entropy numbers to be \emph{optimal} if there is a corresponding lower resp.\ upper bound $(y_n)_{n\geq 1}$ with $x_n\asymp y_n$.

\begin{thm}[Bound for $p<q$]\label{thm:p_leq_q}
Let $0 < p < q\leq \infty$ with $\sfrac{1}{p} = \sfrac{1}{q}+\sfrac{1}{s}$ and $\sigma=(\sigma_k)_{k\geq 1}$ be a sequence with $\sigma_k>0$ and $\sigma_k\searrow 0$. Then the entropy numbers of the diagonal operator $D_\sigma:\ell_p\to\ell_q$ satisfy
\begin{equation}\label{eq:upper_bound_p_leq_q}
\entropy{n}{D_\sigma} 
\preccurlyeq \sup_{k\geq 1}\, k^{-\sfrac{1}{s}} \biggl(\frac{(\sigma_1 + k^{\sfrac{1}{s}}\sigma_k)\cdot\ldots\cdot(\sigma_k + k^{\sfrac{1}{s}}\sigma_k)}{n} \biggr)^{\sfrac{1}{k}}\eqspace.
\end{equation}
If, in addition, there is a real number $b>1$ with 
\begin{equation}\label{eq:exp}\tag{EXP}
\sup_{k\leq n}\frac{\sigma_n b^n}{\sigma_k b^k} < \infty
\end{equation}
then the bound in \eqref{eq:upper_bound_p_leq_q} is optimal and coincides with
\[ 
\entropy{n}{D_\sigma} 
\asymp \sup_{k\geq 1}\,k^{-\sfrac{1}{s}}\biggl(\frac{\sigma_1\cdot\ldots\cdot\sigma_k}{n}\biggr)^{\sfrac{1}{k}}\eqspace.
\]
\end{thm}

Note that the supremum in (\EXP) is taken over all tuples $(n,k)\in\N^2$ with $k\leq n$. Moreover, (\EXP) implies $\sigma_n \preccurlyeq b^{-n}$ and is independent of $p$ and $q$.

To treat the case $p>q$ we recall that the diagonal operator $D_\sigma$ is well-defined if and only if $\sigma\in\ell_r$ with $\sfrac{1}{q}=\sfrac{1}{p}+\sfrac{1}{r}$. For this reason we restricted our considerations in this case to $\sigma\in\ell_r$ and define the \emph{tail sequence} for $k\geq 1$
\begin{equation}\label{eq:tail_def}
\tau_k \coloneqq \Bigl(\sum_{n=k}^\infty \sigma_n^r\Bigr)^{\sfrac{1}{r}}\eqspace.
\end{equation}

\begin{thm}[Bound for $p>q$]\label{thm:p_geq_q}
Let $0< q<p \leq \infty$ with $\sfrac{1}{q} = \sfrac{1}{p} + \sfrac{1}{r}$ and $\sigma=(\sigma_k)_{k\geq 1}\in\ell_r$ be a sequence with $\sigma_k>0$ and $\sigma_k\searrow 0$. Then the entropy numbers of the diagonal operator $D_\sigma:\ell_p\to\ell_q$ satisfy
\begin{equation}\label{eq:upper_bound_p_geq_q}
\entropy{n}{D_\sigma} \preccurlyeq\sup_{k\geq 1} \biggl(\frac{(\tau_k + k^{\sfrac{1}{r}}\sigma_1)\cdot\ldots\cdot(\tau_k + k^{\sfrac{1}{r}}\sigma_k)}{n}\biggr)^{\sfrac{1}{k}}
\eqspace.
\end{equation}
Moreover, under each of the following additional assumptions the bound in \eqref{eq:upper_bound_p_geq_q} is optimal:
\begin{enumerate}
\item\label{it:alp} Assumption~(\ALP): $\tau_n\preccurlyeq \sigma_n n^{\sfrac{1}{r}}$. In this case the bound in \eqref{eq:upper_bound_p_geq_q} coincides with
\[ 
\entropy{n}{D_\sigma} 
\asymp \sup_{k\geq 1}\,k^{\sfrac{1}{r}}\Bigl(\frac{\sigma_1\cdot\ldots\cdot\sigma_k}{n}\Bigr)^{\sfrac{1}{k}}\eqspace.
\]
\item\label{it:amp} Assumption~(\AMP): $\tau_n\succcurlyeq \sigma_n n^{\sfrac{1}{r}}$. In this case the bound in \eqref{eq:upper_bound_p_geq_q} coincides with
\[ 
\entropy{n}{D_\sigma} 
\asymp \tau_{\lfloor\log_2(n)\rfloor +1}\eqspace.
\]
\end{enumerate}
\end{thm}

According to Part~\ref{it:seq:tail:alp} of Lemma~\ref{lem:seq:tail} the Condition~(\ALP) implies $\sigma_n\preccurlyeq n^{-\alpha}$ for some $\alpha>\sfrac{1}{r}$. Moreover, Part~\ref{it:seq:tail:amp} of Lemma~\ref{lem:seq:tail} says that the Condition~(\AMP) is equivalent to $\tau_n\asymp\tau_{2n}$ and according to Lemma~\ref{lem:seq:doubling} this implies $\tau_n\succcurlyeq n^{-\alpha}$ for some $\alpha>0$. Furthermore, from Part~\ref{it:seq:exp:tail} of Lemma~\ref{lem:seq:exp} we get (\EXP)$\subseteq$(\ALP) and (\EXP)$\cap$(\AMP)$=\emptyset$.

Let us now compare our results to the bounds previously obtained in the literature. 
Since essentially all previously established results on the entropy (or covering)  numbers of $D_\sigma$, see e.g.\ \cite{KoTi1961,Mi1961,Ma1974,Ol1978,Ca1981a,K2001a} and the references therein, are contained in \cite{K2005,K2008,CaRu2014}, we restrict our comparison to the latter three articles.

In case of $p<q$ the most general entropy bounds are derived by K\"uhn in \cite{K2005}. Namely, he obtained optimal bounds under each of the following set of assumptions:
\begin{enumerate}
\item\label{it:kue_i} polynomial: $\sup_{k\leq n}\frac{\sigma_n n^\alpha}{\sigma_k k^\alpha} < \infty$ for some $\alpha>0$ and $\sigma_n\asymp\sigma_{2n}$,
\item\label{it:kue_ii} fast logarithmic: $\sup_{k\leq n} \frac{\sigma_n}{\sigma_k} \bigl(\frac{1+\log n}{1+\log k}\bigr)^{\sfrac{1}{s}}<\infty$ and $\sigma_{n^2}\asymp\sigma_n$,
\item\label{it:kue_iii} slow logarithmic: $\inf_{k\leq n} \frac{\sigma_n}{\sigma_k} \bigl(\frac{1+\log n}{1+\log k}\bigr)^{\sfrac{1}{s}}>0$.
\end{enumerate}
Note that Scenario~\ref{it:kue_i} and \ref{it:kue_ii} both exclude sequences that decrease too slow as well as sequences that decrease too fast. In contrast, \ref{it:kue_iii} only excludes sequences that decrease too fast. In comparison, the optimal bounds we obtain in Theorem \ref{thm:p_leq_q} require sequences that decay at least exponentially in the sense of (\EXP). Since all of the Scenarios \ref{it:kue_i}--\ref{it:kue_iii} imply $\sigma_n\asymp\sigma_{2n}$, we easily see that they all exclude (\EXP), that is, (\EXP) is not covered by the results in \cite{K2005}. 

In case of $p>q$, \cite{K2005} also provides optimal bounds for sequences $\sigma$ satisfying
\[
\sup_{k\leq n}\frac{\sigma_n n^\alpha}{\sigma_k k^\alpha} < \infty
\]
for some $\alpha>\sfrac{1}{r}$ and $\sigma_n\asymp\sigma_{2n}$. According to Lemma~\ref{lem:seq:tail} the combination of both assumptions is equivalent to the combination of (\AMP) \emph{and} (\ALP), i.e.\ $\tau_n\asymp\sigma_n n^{\sfrac{1}{r}}$. In \cite{K2008}, K\"uhn generalizes the results of \cite{K2005} by establishing optimal bounds under Assumption~(\AMP), only. Consequently, Theorem~\ref{thm:p_geq_q} recovers the upper bounds of \cite{K2008} and additionally provides optimal bounds for sequences $\sigma$ that only satisfy (\ALP).

Table \ref{table:examples} lists three types of sequences $\sigma$ that are not covered by the literature, but for which we obtain optimal bounds. Compared to \cite{K2005,K2008}, another advantage of our results is that they actually provide bounds for \emph{all} $p\not=q$ and \emph{all} sequences $\sigma$. However, in some cases the question of optimality is not answered yet.

\newcommand{\notcontained}{no}
\newcommand{\contained}{yes}
\begin{table}[t]\label{table:examples} \centering
\resizebox{0.95\textwidth}{!}{
\begin{tabular}{cc|ccc}
$\sigma_n \asymp $ & $\tau_n \asymp $ & (\AMP)  & (\ALP) & (\EXP)\\
\hline
\hline
$\exp\bigl(-a\log^\lambda(n)\bigr)$ & 
$\sigma_n\, n^{\sfrac{1}{r}} \log^{\sfrac{(1-\lambda)}{r}}(n)$ & 
\notcontained &
\contained\ if $\lambda>1$  & 
\notcontained\\
\hline
$\exp\bigl(-an^\lambda\bigr)$ & 
$\sigma_n\, n^{\sfrac{(1-\lambda)_+}{r}}$ & 
\notcontained & 
\contained & 
\contained\ if $\lambda\geq 1$\\
\hline
$\exp\bigl(-ae^{\lambda n}\bigr)$ &
$\sigma_n$ &
\notcontained &
\contained &
\contained\\
\end{tabular}
} 
\caption{Three types of sequences for which our results provide optimal bounds and which are not covered by the existing literature. For all examples we assume $a>0$ and $\lambda > 0$. In addition, the conditions (\AMP) and (\ALP) are only considered in the case $p>q$, whereas (\EXP) is actually independent of $p$ and $q$. Note some subtleties of the first example: For $\lambda=1$ it reduces to a plain polynomial decay, which is already well understood. Moreover, for $\lambda<1$ the operator $D_\sigma$ is not even bounded in case of $p>q$. Finally, for $\lambda<1$ and $p<q$, K\"uhn \cite{K2005} leaves the behavior of $\entropy{n}{D_\sigma}$  as an open question, which our results cannot address, either.
}
\end{table}

There is another strand of research, see e.g.~\cite{Ca1981a, CaRu2014}, that describes the asymptotic behavior of the entropy numbers in terms of \emph{(generalized) Lorentz spaces}. The most general result in this direction is \cite[Corollary~1.2]{CaRu2014}:
\[ 
\sigma \in \ell_{t,v,\varphi} \quad\iff\quad
\entropy{2^{n-1}}{D_\sigma} \in \ell_{u,v,\varphi}\eqspace, 
\]
where $\ell_{u,v,\varphi}$ is a generalized Lorentz space with slowly varying function
$\varphi$, see \cite[Section~2]{CaRu2014} for a definition, and the parameters satisfy $1\leq p,q\leq \infty$, $0<t,v\leq\infty$, $\sfrac{1}{t} > (\sfrac{1}{q}-\sfrac{1}{p})_+$, and $\sfrac{1}{u} = \sfrac{1}{t} - (\sfrac{1}{q} - \sfrac{1}{p})$. Note that the implication ($\Leftarrow$) is contained in Lemma~\ref{lem:entropy:lower_bound} and ($\Rightarrow$) is contained in Theorem~\ref{thm:p_geq_q} if $p>q$ and $v=\infty$.

Finally, many results previously obtained in the literature are based on the operator ideal theory and a dyadic splitting of the diagonal operator, see e.g.\ \cite{Ca1981a,K2001a,CaRu2014}. This approach reduces the problem of bounding $\entropy{n}{D_\sigma}$ to the estimation of entropy numbers of embeddings between finite dimensional sequence spaces. In order to bound the entropy numbers of these finite dimensional embeddings advanced bounds with a good so-called preasymptotic behavior are needed. Such bounds can be found e.g.\ in \cite{Sc1984,EdTr1996,GuLi2000,K2001} and are often based on sophisticated combinatoric arguments and interpolation theory. In contrast, our results are based on a single splitting of the diagonal operator and a simple bound for finite dimensional diagonal operators. The latter bound has no good preasymptotic behavior but it is easily proven by a plain volume argument.

\subsection*{Acknowledgment} 
I am especially grateful to Ingo Steinwart and Thomas K\"uhn for carefully proofreading preliminary versions of this manuscript and pointing out some errors. 
Moreover, I am very appreciative to the anonymous referees for their constructive comments and suggestions for improvement.
Finally, I am thankful to the International Max Planck Research School for Intelligent Systems (IMPRS-IS) for its support.
%
%



\section{Proofs}


Before we prove the main theorems we summarize some preparatory results. Because we will reduce the investigation of diagonal operators to the case of diagonal operators on $\R^k$ we will include this case in the following. To this end, we consider sequences over an index set $I\subseteq\N$ and define, for $0 < p \leq \infty$, the sequence space $\ell_p(I) \coloneqq \{x=(x_i)_{i\in I}\in\R^I:\ \|x\|_{\ell_p(I)} < \infty\}$ with norm
\[ 
\|x\|_{\ell_p(I)} \coloneqq \biggl(\sum_{i\in I} |x_i|^p\biggr)^{\sfrac{1}{p}}
\]
and closed unit ball $\uball{\ell_p(I)}$. With this notation we have $\ell_p=\ell_p(\N)$ and for $k\geq 1$ we introduce the abbreviation $\ell_p^k \coloneqq \ell_p(\{1,\ldots,k\})$. It is well-known that
\[ 
\|x + y\|_{\ell_p(I)} \leq \quasi{p} \bigl(\|x\|_{\ell_p(I)} + \|y\|_{\ell_p(I)}\bigr)
\]
holds for all $x,y\in\ell_p(I)$ with $\quasi{p} \coloneqq \max\{1, 2^{\sfrac{1}{p} - 1}\}$. Consequently, $\ell_p(I)$ is a quasi-Banach space for all $0<p\leq\infty$ and $\ell_p(I)$ is a Banach space iff $1\leq p\leq \infty$.

In the following, we fix $0 < p,q \leq \infty$, a sequence $\sigma=(\sigma_i)_{i\in I}\in\R^I$, and the diagonal operator $D_\sigma:\ell_p(I)\to\ell_q(I)$ defined by $D_\sigma(x_i)_{i\in I} \coloneqq (\sigma_i x_i)_{i\in I}$.  
As a consequence of H\"older's inequality the operator norm of $D_\sigma$ satisfies
\begin{equation}\label{eq:prep:op_norm}
\|D_\sigma\| = \begin{cases}
\|\sigma\|_{\ell_r(I)}, & p> q,\ \sfrac{1}{q} = \sfrac{1}{p} + \sfrac{1}{r}\\
\|\sigma\|_{\ell_\infty(I)}, &p\leq q \eqspace.
\end{cases}
\end{equation}

Next, we introduce a concept related to entropy numbers. For $\varepsilon>0$ the \emph{covering number} of $D_\sigma$ is defined by
\[
\covering{\varepsilon}{D_\sigma} \coloneqq \min\Bigl\{n\geq 1:\, \exists y_1,\ldots,y_n\in \ell_q(I)\text{ with }D_\sigma\uball{\ell_p(I)}\subseteq\bigcup_{i=1}^n y_i + \varepsilon\uball{\ell_q(I)}\Bigr\}
\eqspace.
\]
The next result establishes a comparison between covering and entropy numbers.

\begin{lem}\label{lem:prep:covering_entropy}
Let $0< p,q\leq\infty$, $(a_k)_{k\geq 1}$ be a positive sequence and $D_\sigma:\ell_p\to\ell_q$ be a diagonal operator with $\|D_\sigma\| < \infty$. If the covering number estimate
\begin{equation}\label{eq:prep:covering_bound}
\covering{\varepsilon}{D_\sigma}\leq \sup_{k\geq 1} a_k\Bigl(\frac{1}{\varepsilon}\Bigr)^k
\end{equation}
holds for all $0<\varepsilon<\entropy{1}{D_\sigma}$, then for all $n\geq 1$ the $n$-th entropy number satisfies
\[ 
\entropy{n}{D_\sigma}\leq \sup_{k\geq 1} \Bigl(\frac{a_k}{n}\Bigr)^{\sfrac{1}{k}}\eqspace.
\]
\end{lem}

Note that $\sfrac{\|D_\sigma\|}{\quasi{q}} \leq \entropy{1}{D_\sigma} \leq \|D_\sigma\|$ is satisfied, see e.g.\ \cite[p.~11]{CaSt1990} for the Banach space case and e.g.\ \cite[Lemma~1 on p.~7]{EdTr1996} for the general case. Consequently, in the Lemma~\ref{lem:prep:covering_entropy} it is sufficient to check \eqref{eq:prep:covering_bound} for all $0<\varepsilon<\|D_\sigma\|$.

\begin{proof}
Let $n\geq 1$ be a natural number. If $\entropy{n}{D_\sigma}=0$ holds then $D_\sigma = 0$ is the zero operator and there is nothing to prove. In the following we assume $\entropy{n}{D_\sigma}>0$ and choose $0<\varepsilon<\entropy{n}{D_\sigma}$. By the definition of entropy and covering numbers we have $n<\covering{\varepsilon}{D_\sigma}$. Moreover, by our assumption there is, for every $\delta>0$, a $k_\delta\geq 1$ with
\[
n\leq \covering{\varepsilon}{D_\sigma} \leq (1+\delta)\, a_{k_\delta}\Bigl(\frac{1}{\varepsilon}\Bigr)^{k_\delta}\eqspace.
\]
This implies
\[ 
\varepsilon \leq  \Bigl(\frac{(1+\delta)\,a_{k_\delta}}{n}\Bigr)^{\sfrac{1}{k_\delta}} \leq (1+\delta)\, \Bigl(\frac{a_{k_\delta}}{n}\Bigr)^{\sfrac{1}{k_\delta}} \leq (1+\delta)\sup_{k\geq 1}\Bigl(\frac{a_k}{n}\Bigr)^{\sfrac{1}{k}}\eqspace.
\]
Letting $\delta\searrow 0$ and $\varepsilon\nearrow\entropy{n}{D_\sigma}$ we get the assertion.
\end{proof}

In the following, $\lambda^k$ denotes the $k$-dimensional Lebesgue measure.

\begin{lem}\label{lem:prep:vol_arg}
Let $0< p,q \leq \infty$, $k\geq 1$ and $\sigma_1,\ldots,\sigma_k>0$. Then for all $\varepsilon
>0$ the diagonal operator $D_\sigma :\ell_p^k \to \ell_q^k$ satisfies
\begin{equation}\label{eq:prep:vol_arg}
\covering{2 \varepsilon}{D_\sigma} \leq
(2 \quasi{p})^k\, \frac{\lambda^k(\uball{\ell_p^k})}{\lambda^k(\uball{\ell_q^k})}\, \Bigl(\|\id_{q,p}^k\| + \quasi{q}\frac{\sigma_1}{\varepsilon}\Bigr)\cdot\ldots\cdot\Bigl(\|\id_{q,p}^k\| + \quasi{q}\frac{\sigma_k}{\varepsilon}\Bigr)\eqspace,
\end{equation}
where $\id_{q,p}^k:\ell_q^k\to\ell_p^k$ denotes the identity operator.
\end{lem}

In case of $p=q$ the bound in \eqref{eq:prep:vol_arg} originates from \citet[Hilfsatz~2]{Ol1978}. Furthermore, note that the proof of \citet[Theorem~XVI]{KoTi1961} contains the case $p=q=2$ and $\sigma_n=n^{-\alpha}$.

\begin{proof}
For this proof we use \emph{packing numbers}, which for $\varepsilon>0$ are defined by
\[ 
\packing{\varepsilon}{D_\sigma} \coloneqq \max\Bigl\{n\geq 1:\ \exists y_1,\ldots,y_n\in D_\sigma\uball{\ell_p^k}\text{ with }\|y_i-y_j\|_{\ell_q^k}>2\varepsilon\ \forall i\not=j\Bigr\}\eqspace.
\]
Recall that $\packing{2\quasi{q}\varepsilon}{D_\sigma} \leq \covering{2\varepsilon}{D_\sigma} \leq \packing{\varepsilon}{D_\sigma}$ holds for all $\varepsilon>0$, see e.g.\ \cite[Theorem~IV]{KoTi1961} for the Banach space case. Therefore it is enough to prove that $\packing{\varepsilon}{D_\sigma}$ is bounded by the right hand side of \eqref{eq:prep:vol_arg}. 

Now, for $\varepsilon>0$ and $n \coloneqq \packing{\varepsilon}{D_\sigma}$ we choose $x_1,\ldots, x_n\in D_\sigma\uball{\ell_p^k}$ with $\|x_i - x_j\|_{\ell_q^k}>2\varepsilon$ for all $i\not=j$. Then $x_i + \sfrac{\varepsilon}{\quasi{q}}\, \uball{\ell_q^k}$ are disjoint sets contained in $D_\sigma\uball{\ell_p^k} + \sfrac{\varepsilon}{\quasi{q}}\,\uball{\ell_q^k}$. Hence their volume satisfies
\begin{equation}\label{eq:prep:vol_arg:prep}
n(\sfrac{\varepsilon}{\quasi{q}})^k\lambda^k(\uball{\ell_q^k}) = \lambda^k\Bigl(\bigcup_{i=1}^n \bigl(x_i + \sfrac{\varepsilon}{\quasi{q}}\,\uball{\ell_q^k}\bigr)\Bigr) \leq \lambda^k(D_\sigma\uball{\ell_p^k} + \sfrac{\varepsilon}{\quasi{q}}\,\uball{\ell_q^k})\eqspace.
\end{equation}

Before we continue to estimate \eqref{eq:prep:vol_arg:prep} we prove the following auxiliary result: For a second operator $D_\omega:\ell_p^k\to\ell_q^k$ with $\omega_i>0$ for all $i=1,\ldots, k$ we have
\begin{equation}\label{eq:prep:vol_arg:aux}
D_\sigma\uball{\ell_p^k} + D_\omega\uball{\ell_p^k}\subseteq 2\quasi{p} D_{\sigma+\omega}\uball{\ell_p^k}\eqspace.
\end{equation}
Since $D_{\sigma+\omega}$ is invertible \eqref{eq:prep:vol_arg:aux} is equivalent to $D_{\sigma+\omega}^{-1}(D_\sigma\uball{\ell_p^k} + D_\omega\uball{\ell_p^k})\subseteq 2\quasi{p} \uball{\ell_p^k}$. Now, to show \eqref{eq:prep:vol_arg:aux} we fix $x,y\in\uball{\ell_p^k}$ and observe
\begin{align*}
\|D_{\sigma+\omega}^{-1}(D_\sigma x + D_\omega y)\|_{\ell_p^k}
&\leq \quasi{p}\|D_{\sigma+\omega}^{-1} D_\sigma x\|_{\ell_p^k} + \quasi{p}\|D_{\sigma+\omega}^{-1} D_\omega y\|_{\ell_p^k}\\
&\leq \quasi{p}\|D_{\sigma + \omega}^{-1} D_\sigma\| + \quasi{p}\|D_{\sigma + \omega}^{-1}D_\omega\|
\eqspace.
\end{align*}
Since $D_{\sigma + \omega}^{-1} D_\sigma$ is an operator from $\ell_p^k$ to $\ell_p^k$ the operator norm is given by $\|D_{\sigma + \omega}^{-1} D_\sigma\| = \max_{i=1,\ldots,k}\frac{\sigma_i}{\sigma_i + \omega_i} \leq 1$. Analogously we have $\|D_{\sigma + \omega}^{-1} D_\omega\| = \max_{i=1,\ldots,k}\frac{\omega_i}{\sigma_i + \omega_i} \leq 1$ and therefore \eqref{eq:prep:vol_arg:aux} is proven. 

By the definition of the operator norm we have $\uball{\ell_q^k}\subseteq\|\id_{q,p}^k\|\uball{\ell_p^k}$. Together with \eqref{eq:prep:vol_arg:aux} we get
\[
D_\sigma\uball{\ell_p^k} + \sfrac{\varepsilon}{\quasi{q}}\,\uball{\ell_q^k}
\subseteq D_\sigma\uball{\ell_p^k} + \sfrac{\varepsilon}{\quasi{q}}\,\|\id_{q,p}^k\|\uball{\ell_p^k} 
\subseteq 2\quasi{p} D_{\sigma + \sfrac{\varepsilon}{\quasi{q}}\,\|\id_{q,p}^k\|}\uball{\ell_p^k}\eqspace. 
\]
Continuing estimate~\eqref{eq:prep:vol_arg:prep} with this inclusion yields \eqref{eq:prep:vol_arg}.
\end{proof}

\subsection{Entropy Bounds}

\newcommand{\proj}[2]{P_{#1}^{#2}}
\newcommand{\emb}[2]{I_{#1}^{#2}}
\newcommand{\diagF}[3]{D_{#1,#2}^{#3}}
\newcommand{\diagS}[3]{\emb{q}{k}\diagF{p}{q}{k}\proj{p}{k}}

In this subsection we provide lower and upper bounds on the entropy numbers. To this end, we define, for $k\geq 1$, the auxiliary operators
\begin{align*}
\diagF{p}{q}{k}:\ell_p^k\to\ell_q^k,\ (x_n)_{n=1}^k &\mapsto (\sigma_1 x_1,\ldots, \sigma_k x_k)\eqspace,\\ 
\proj{p}{k}:\ell_p\to\ell_p^k,\ (x_n)_{n\geq 1} &\mapsto (x_1,\ldots,x_k)\eqspace,\\
\emb{p}{k}:\ell_p^k\to\ell_p,\ (x_n)_{n=1}^k &\mapsto (x_1,\ldots, x_k,0,0,\ldots)\eqspace.
\end{align*}
Note that these operators satisfy $\diagF{p}{q}{k} = \proj{q}{k}D_\sigma\emb{p}{k}$ and $\|\emb{p}{k}\|=\|\proj{p}{k}\| = 1$.

\begin{lem}[Lower Bound]\label{lem:entropy:lower_bound}
Let $0 < p,q \leq \infty$ and $\sigma=(\sigma_k)_{k\geq 1}$ with $\sigma_k>0$ and $\sigma_k\searrow 0$ such that the diagonal operator $D_\sigma:\ell_p \to \ell_q$ is bounded. Then for all $n\geq 1$ the $n$-th entropy number satisfies
\[
\entropy{n}{D_\sigma} \geq \sup_{k\geq 1}\biggl(\frac{\lambda^k(\uball{\ell_p^k})}{\lambda^k(\uball{\ell_q^k})} \,\frac{\sigma_1\cdot\ldots\cdot\sigma_k}{n}\biggr)^{\sfrac{1}{k}}\eqspace.
\]
\end{lem}

Note that this lower bound holds without any additional assumption on $\sigma$. Moreover, a combination of \cite[Equation~(1.17)]{Pi1989} with Stirling's formula yields
\begin{equation}\label{eq:entropy:vol_ratio}
\biggl(\frac{\lambda^k(\uball{\ell_p^k})}{\lambda^k(\uball{\ell_q^k})}\biggr)^{\sfrac{1}{k}} 
\asymp k^{\sfrac{1}{q} - \sfrac{1}{p}}\eqspace.
\end{equation}

\begin{proof}
By the multiplicativity of entropy numbers, see e.g.\ \cite[p.~11]{CaSt1990} for the Banach space case and e.g.\ \cite[Lemma~1 on p.~7]{EdTr1996} for the general case, we find $\entropy{n}{\diagF{p}{q}{k}} = \entropy{n}{\proj{q}{k}D_\sigma\emb{p}{k}} \leq \entropy{n}{D_\sigma}$, and hence it remains to give a lower bound for $\entropy{n}{\diagF{p}{q}{k}}$. To this end, choose for $\varepsilon>\entropy{n}{\diagF{p}{q}{k}}$ some $x_1,\ldots, x_n\in\R^k$ with $D_\sigma\uball{\ell_p^k}\subseteq\bigcup_{i=1}^n(x_i + \varepsilon \uball{\ell_q^k})$. Consequently, the volume of these sets satisfy
\[ 
\sigma_1\cdot\ldots\cdot\sigma_k\lambda^k(\uball{\ell_p^k}) = \lambda^k(D_\sigma\uball{\ell_p^k})\leq \sum_{i=1}^n \lambda^k(x_i + \varepsilon \uball{\ell_q^k}) = n\varepsilon^k\lambda^k(\uball{\ell_q^k})\eqspace,
\]
and hence we find
\[ 
\varepsilon \geq \biggl(\frac{\lambda^k(\uball{\ell_p^k})}{\lambda^k(\uball{\ell_q^k})} \,\frac{\sigma_1\cdot\ldots\cdot\sigma_k}{n}\biggr)^{\sfrac{1}{k}}\eqspace.
\]
Letting $\varepsilon\searrow\entropy{n}{\diagF{p}{q}{k}}$ and taking the supremum over $k\geq 1$ we get the claim.
\end{proof}

Since the upper bounds in \eqref{eq:upper_bound_p_leq_q} and \eqref{eq:upper_bound_p_geq_q} are based on the same decomposition we first introduce this decomposition. To this end, recall that the covering numbers have an additivity and multiplicativity property analogously to the entropy numbers, see e.g.\ \cite[p.~11]{CaSt1990} for the Banach space case and e.g.\ \cite[Lemma~1 on p.~7]{EdTr1996} for the general case. Using these properties yields
\begin{align*}
\covering{\quasi{q}\varepsilon}{D_\sigma}
&= \covering[Big]{\quasi{q}\varepsilon}{\diagS{p}{q}{k} + (D_\sigma-\diagS{p}{q}{k})}\\
&\leq \covering[Big]{\sfrac{\varepsilon}{2}}{\diagS{p}{q}{k}}\cdot \covering[Big]{\sfrac{\varepsilon}{2}}{D_\sigma-\diagS{p}{q}{k}}\\
&\leq \covering[Big]{\sfrac{\varepsilon}{2}}{\diagF{p}{q}{k}}\cdot \covering[Big]{\sfrac{\varepsilon}{2}}{D_\sigma-\diagS{p}{q}{k}}\eqspace.
\end{align*}
In the following, we will choose a suitable $k$ with $\|D_\sigma - \diagS{p}{q}{k}\| \leq \sfrac{\varepsilon}{2}$. Since in this case we have $\covering{\sfrac{\varepsilon}{2}}{D_\sigma - \diagS{p}{q}{k}} = 1$ the estimate above reduces to
\begin{equation}\label{eq:entropy:prep}
\covering{\quasi{q}\varepsilon}{D_\sigma} \leq \covering{\sfrac{\varepsilon}{2}}{\diagF{p}{q}{k}}\eqspace.
\end{equation}
Let us first treat the case $p<q$. 

\begin{lem}\label{lem:entropy:p_leq_q}
Let $0< p<q \leq \infty$ with $\sfrac{1}{p} = \sfrac{1}{q} + \sfrac{1}{s}$ and $\sigma=(\sigma_k)_{k\geq 1}$ with $\sigma_k>0$ and $\sigma_k\searrow 0$. Then for all $n\geq 1$ the diagonal operator $D_\sigma:\ell_p \to \ell_q$ satisfies
\[
\entropy{n}{D_\sigma} 
\leq 4\quasi{p}\quasi{q}\, \sup_{k\geq 1} \biggl(\frac{\lambda^k(\uball{\ell_p^k})}{\lambda^k(\uball{\ell_q^k})}\,\frac{(2\quasi{q}\sigma_1 + k^{\sfrac{1}{s}}\sigma_k)\cdot\ldots\cdot(2\quasi{q}\sigma_k + k^{\sfrac{1}{s}}\sigma_k)}{n} \biggr)^{\sfrac{1}{k}}\eqspace.
\]
\end{lem}

\begin{proof}
For every $0<\sfrac{\varepsilon}{2}< \|D_\sigma\|=\sigma_1$, there is a $k\geq 1$ with $\sigma_{k+1}\leq\sfrac{\varepsilon}{2} < \sigma_k$. Equation~\eqref{eq:prep:op_norm} gives us $\|D_\sigma - \diagS{p}{q}{k}\| = \sigma_{k+1} \leq\sfrac{\varepsilon}{2}$. Using Equation~\eqref{eq:entropy:prep} with this $k$, Lemma~\ref{lem:prep:vol_arg}, and $\|\id_{q,p}^k\| = k^{\sfrac{1}{s}}$ we get
\begin{align*}
\covering{\quasi{q}\varepsilon}{D_\sigma} 
&\leq \covering{\sfrac{\varepsilon}{2}}{\diagF{p}{q}{k}}\\
&\leq (2\quasi{p})^k \frac{\lambda^k(\uball{\ell_p^k})}{\lambda^k(\uball{\ell_q^k})} \Bigl(k^{\sfrac{1}{s}} + \frac{4\quasi{q}\sigma_1}{\varepsilon}\Bigr)\cdot\ldots\cdot\Bigl(k^{\sfrac{1}{s}} + \frac{4\quasi{q}\sigma_k}{\varepsilon}\Bigr)
\eqspace.
\end{align*}
Using $k^{\sfrac{1}{s}}<\sfrac{2\sigma_k k^{\sfrac{1}{s}}}{\varepsilon}$ and taking the supremum over $k\geq 1$ gives
\[ 
\covering{\quasi{q}\varepsilon}{D_\sigma} 
\leq \sup_{k\geq 1} \biggl\{\frac{\lambda^k(\uball{\ell_p^k})}{\lambda^k(\uball{\ell_q^k})} \bigl(\sigma_k k^{\sfrac{1}{s}} + 2\quasi{q}\sigma_1\bigr)\cdot\ldots\cdot\bigl(\sigma_k k^{\sfrac{1}{s}} + 2\quasi{q}\sigma_k\bigr) \Bigl(\frac{4\quasi{p}}{\varepsilon}\Bigr)^k\biggr\}\eqspace.
\]
Finally, Lemma~\ref{lem:prep:covering_entropy} yields the assertion.
\end{proof}

\begin{lem}\label{lem:entropy:p_geq_q}
Let $0< q < p \leq \infty$ with $\sfrac{1}{q} = \sfrac{1}{p} + \sfrac{1}{r}$, $\sigma=(\sigma_k)_{k\geq 1}\in\ell_r$ with $\sigma_k>0$ and $\sigma_k\searrow 0$, and $\tau$ the tail sequence defined by \eqref{eq:tail_def}. Then for all $n\geq 1$ the diagonal operator $D_\sigma:\ell_p \to \ell_q$ satisfies
\[
\entropy{n}{D_\sigma}
\leq  4\quasi{p}\quasi{q}\,\sup_{k\geq 1} \biggl(\frac{(\tau_k + 2\quasi{p}k^{\sfrac{1}{r}}\sigma_1)\cdot\ldots\cdot(\tau_k + 2\quasi{p}k^{\sfrac{1}{r}}\sigma_k)}{n}\biggr)^{\sfrac{1}{k}}\eqspace.
\]
\end{lem}

\begin{proof}
For every $0 < \sfrac{\varepsilon}{2} < \|D_\sigma\| = \tau_1$, there is a $k\geq 1$ with $\tau_{k+1}\leq\sfrac{\varepsilon}{2}<\tau_k$. Equation~\eqref{eq:prep:op_norm} gives us $\|D_\sigma - \diagS{p}{q}{k}\| = \tau_{k+1} \leq\sfrac{\varepsilon}{2}$. Using Equation~\eqref{eq:entropy:prep} with this $k$, the decomposition $\diagF{p}{q}{k} = \id_{p,q}^k\circ\diagF{p}{p}{k}$, and $\|\id_{p,q}^k\| = k^{\sfrac{1}{r}}$ we get
\[ 
\covering{\quasi{q}\varepsilon}{D_\sigma} 
\leq \covering{k^{-\sfrac{1}{r}}\sfrac{\varepsilon}{2}}{\diagF{p}{p}{k}}\cdot\covering{k^{\sfrac{1}{r}}}{\id_{p,q}^k} 
= \covering{k^{-\sfrac{1}{r}}\sfrac{\varepsilon}{2}}{\diagF{p}{p}{k}}\eqspace.
\]
Using Lemma~\ref{lem:prep:vol_arg} and $1<\sfrac{2\tau_k}{\varepsilon}$ gives
\begin{align*}
\covering{\quasi{q}\varepsilon}{D_\sigma} 
&\leq (2\quasi{p})^k \Bigl(1+\frac{4\quasi{p}k^{\sfrac{1}{r}}\sigma_1}{\varepsilon}\Bigr)\cdot\ldots\cdot\Bigl(1+\frac{4\quasi{p}k^{\sfrac{1}{r}}\sigma_k}{\varepsilon}\Bigr)\\
&\leq  \bigl(\tau_k + 2\quasi{p}k^{\sfrac{1}{r}}\sigma_1\bigr)\cdot\ldots\cdot\bigl(\tau_k + 2\quasi{p}k^{\sfrac{1}{r}}\sigma_k\bigr)\Bigl(\frac{4\quasi{p}}{\varepsilon}\Bigr)^k\eqspace.
\end{align*}
Finally, taking the supremum over $k$ and using Lemma~\ref{lem:prep:covering_entropy} gives the assertion.
\end{proof}

\subsection{Optimality}

\begin{proof}[Proof of Theorem~\ref{thm:p_leq_q}]
The upper bound in \eqref{eq:upper_bound_p_leq_q} is a consequence of Lemma~\ref{lem:entropy:p_leq_q} and Equation~\eqref{eq:entropy:vol_ratio}. It remains to prove the optimality under (\EXP). To this end, we continue the estimate of the upper bound as follows
\[ 
\entropy{n}{D_\sigma} 
\preccurlyeq \sup_{k\geq 1}\,k^{-\sfrac{1}{s}}\biggl(\frac{\sigma_1\cdot\ldots\cdot\sigma_k}{n}\biggr)^{\sfrac{1}{k}} \biggl(\Bigl(1 + \frac{k^{\sfrac{1}{s}}\sigma_k}{\sigma_1}\Bigr)\ldots\Bigl(1 + \frac{k^{\sfrac{1}{s}}\sigma_k}{\sigma_k}\Bigr) \biggr)^{\sfrac{1}{k}}\eqspace.
\]
Applying that the geometric mean is bounded by the arithmetic mean as well as the triangle inequality in $\ell_s^k$ yields
\begin{align*}
\biggl(\Bigl(1 + \frac{k^{\sfrac{1}{s}}\sigma_k}{\sigma_1}\Bigr)\ldots\Bigl(1 + \frac{k^{\sfrac{1}{s}}\sigma_k}{\sigma_k}\Bigr) \biggr)^{\sfrac{1}{k}}
&\leq \biggl(\sfrac{1}{k}\sum_{i=1}^k \Bigl(1 + \frac{k^{\sfrac{1}{s}}\sigma_k}{\sigma_i}\Bigr)^s\biggr)^{\sfrac{1}{s}}\\
&\leq \quasi{s} + \quasi{s}\sigma_k\biggl(\sum_{i=1}^k \sigma_i^{-s}\biggr)^{\sfrac{1}{s}}\eqspace.
\end{align*}
According Part~\ref{it:seq:exp:partial} of Lemma~\ref{lem:seq:exp} the right hand side is bounded in $k$ and we get the claimed upper bound. If we combine Lemma~\ref{lem:entropy:lower_bound} with Equation~\eqref{eq:entropy:vol_ratio} we get the corresponding lower bound.
\end{proof}

\begin{proof}[Proof of Theorem~\ref{thm:p_geq_q}]
The upper bound in \eqref{eq:upper_bound_p_geq_q} directly follows from Lemma~\ref{lem:entropy:p_geq_q} and it thus remains to prove the optimality under (\ALP) and (\AMP).

(\ALP) The upper bound \eqref{eq:upper_bound_p_geq_q} can be transformed into
\begin{equation*}
\entropy{n}{D_\sigma} \preccurlyeq \sup_{k\geq 1}\,k^{\sfrac{1}{r}}\Bigl(\frac{\sigma_1\cdot\ldots\cdot\sigma_k}{n}\Bigr)^{\sfrac{1}{k}} \biggl(\Bigl(\frac{\tau_k}{k^{\sfrac{1}{r}}\sigma_1}+ 1 \Bigr)\cdot\ldots\cdot\Bigl(\frac{\tau_k}{k^{\sfrac{1}{r}}\sigma_k} + 1\Bigr)\biggr)^{\sfrac{1}{k}}\eqspace.
\end{equation*}
According to (\ALP) the last factor is bounded in $k$. This yields the claimed upper bound. The claimed lower bound is a consequence of Lemma~\ref{lem:entropy:lower_bound} and \eqref{eq:entropy:vol_ratio}.

(\AMP) Because of Part~\ref{it:seq:tail:amp} of Lemma~\ref{lem:seq:tail} we have $\tau_n\asymp \tau_{2n}$. Hence \citet[Theorem~1]{K2008} yields $\entropy{n}{D_\sigma}\asymp \tau_{\lfloor\log_2(n)\rfloor + 1}$ and it is enough to show that upper bound in \eqref{eq:upper_bound_p_geq_q} is asymptotically bounded by $\tau_{\lfloor\log_2(n)\rfloor + 1}$. According to (\AMP) and Part~\ref{it:seq:doubling:alm_incr} of Lemma~\ref{lem:seq:doubling} applied to $(\tau_n)_{n\geq 1}$ there are constants $c_1,c_2,\beta>0$ with $\sigma_i\leq c_1 \tau_i i^{-\sfrac{1}{r}}$ and $\tau_i \leq c_2 \tau_k k^\beta i^{-\beta}$ for all $k\geq i$. Together we get for $\alpha \coloneqq \sfrac{1}{r} + \beta$
\[ 
\tau_k + k^{\sfrac{1}{r}}\sigma_i 
\leq \tau_k + c_1c_2 \tau_k \frac{k^{\sfrac{1}{r}+\beta}}{i^{\sfrac{1}{r} + \beta}} 
\leq \tau_k \frac{k^{\alpha}}{i^{\alpha}} (1 +c_1c_2)
\]
and all $k\geq i$. Plugging this into the bound in \eqref{eq:upper_bound_p_geq_q} we get
\begin{equation*}
\entropy{n}{D_\sigma}
\preccurlyeq  \sup_{k\geq 1} \biggl(\frac{(\tau_k + k^{\sfrac{1}{r}}\sigma_1)\cdot\ldots\cdot(\tau_k + k^{\sfrac{1}{r}}\sigma_k)}{n}\biggr)^{\sfrac{1}{k}}
\preccurlyeq\sup_{k\geq 1} \frac{\tau_k}{n^{\sfrac{1}{k}}}  \frac{k^\alpha}{(k!)^{\sfrac{\alpha}{k}}}\eqspace.
\end{equation*}
From Stirling's formula we know $(k!)^{\sfrac{1}{k}}\asymp k$. Consequently, we have 
\begin{equation}\label{eq:opt:temp_i}
\entropy{n}{D_\sigma} \preccurlyeq \sup_{k\geq 1} \frac{\tau_k}{n^{\sfrac{1}{k}}}
\end{equation}
and it remains to show, that the right hand side behaves asymptotically like $\tau_{\lfloor\log_2(n)\rfloor + 1}$. To this end, let $c>0$ be the doubling constant of $\tau$, i.e.\ $\tau_{2n}\geq c\tau_n$ for all $n\geq 1$. Without loss of generality we can assume $c<1$ and define $\alpha \coloneqq \frac{\log(2)}{2\log(\sfrac{1}{c})}>0$. For $k\leq \alpha\log_2(n)$ we have 
\[ 
n^{\frac{1}{2k} - \frac{1}{k}} 
= n^{-\frac{1}{2k}} 
\leq \exp\Bigl( - \frac{\log(n)}{2\alpha\log_2(n)} \Bigr)
= c 
\leq \frac{\tau_{2k}}{\tau_k}
\]
and this implies
\begin{equation}\label{eq:opt:temp_ii}
\frac{\tau_k}{n^{\frac{1}{k}}}\leq\frac{\tau_{2k}}{n^{\frac{1}{2k}}}\eqspace.
\end{equation}
A recursive application of this inequality enables us to restrict our supremum to $k > \alpha\log_2(n)$. Moreover, for such $k$ we have
\begin{equation}\label{eq:opt:temp_iii} 
1
\geq n^{-\sfrac{1}{k}} 
= \exp\Bigl(- \frac{\log(n)}{k}\Bigr) 
\geq \exp\Bigl( -\frac{\log(n)}{\alpha\log_2(n)} \Bigr)
= 2^{-\sfrac{1}{\alpha}}\eqspace.
\end{equation}
If we combine \eqref{eq:opt:temp_i}, \eqref{eq:opt:temp_ii}, and \eqref{eq:opt:temp_iii}, then we get
\[ 
\entropy{n}{D_\sigma}
\preccurlyeq \sup_{k\geq 1}\frac{\tau_k}{n^{\sfrac{1}{k}}} 
= \sup_{k>\alpha\log_2(n)} \frac{\tau_k}{n^{\sfrac{1}{k}}} 
\asymp \sup_{k>\alpha\log_2(n)} \tau_k 
= \tau_{\lfloor\alpha\log_2(n)\rfloor + 1}\eqspace.
\]
Finally, an application of Part~\ref{it:seq:doubling:o_varying} of Lemma~\ref{lem:seq:doubling} yields the assertion.
\end{proof}

\appendix



\section{Conditions on Sequences}\label{sec:seq}


In this section we collect some characterizations of the conditions used on the diagonal sequence. Most of them are consequences of the general theory of $\mathcal{O}$-regular varying functions/sequences, but for convenience we include the proofs or give detailed references. These results enable us to compare our findings with \cite{K2005,K2008}. In the following, all supremums $\sup_{k\leq n}$ and infimums $\inf_{k\leq n}$ are taken over all tuples $(n,k)\in\N^2$ with $k\leq n$. 

\begin{lem}[(\EXP) Sequences]\label{lem:seq:exp}
Let $r,s>0$, $\sigma=(\sigma_k)_{k\geq 1}$ with $\sigma_k>0$ and $\sigma_k\searrow 0$, $\tau$ be the tail sequence given by \eqref{eq:tail_def}, and $v_n \coloneqq \bigl(\sum_{k=1}^n \sigma_k^{-s}\bigr)^{\sfrac{1}{s}}$ the partial sum sequence. Then the following statements are equivalent:
\begin{enumerate}
\item\label{it:seq:exp:alm_decr} There is a real number $b>1$ with $\sup_{k\leq n}\frac{\sigma_n b^n}{\sigma_k b^k}<\infty$.
\item\label{it:seq:exp:shifted} There is an $n_0\geq 1$ and a $0<a<1$ with $\sigma_{k+n_0} \leq a\, \sigma_k$ for all $k\geq 1$.
\item\label{it:seq:exp:partial} $\sigma_n\asymp\sfrac{1}{v_n}$.
\item\label{it:seq:exp:tail} $\sigma_n\asymp \tau_n$. 
\end{enumerate}
\end{lem}

Note that Condition~\ref{it:seq:exp:alm_decr} and \ref{it:seq:exp:shifted} are independent of $r>0$ and $s>0$. Consequently, if $\sigma$ satisfies Condition~\ref{it:seq:exp:partial} or \ref{it:seq:exp:tail} for some $s>0$ resp.\ $r>0$ then $\sigma$ satisfies both conditions for all $r,s>0$.

\begin{proof}
\ref{it:seq:exp:alm_decr}$\Rightarrow$\ref{it:seq:exp:partial} For $c \coloneqq \sup_{k\leq n}\frac{\sigma_n b^n}{\sigma_k b^k}<\infty$ we get
\[ 
v_n^s\sigma_n^s
= \sum_{k=1}^n \Bigl(\frac{\sigma_n}{\sigma_k}\Bigr)^s 
\leq c^s \sum_{k=1}^n b^{-s(n-k)}
= c^s \sum_{k=0}^{n-1} b^{-sk}
\leq \frac{(bc)^s}{b^s-1}
\]
for all $n\geq 1$. Moreover, $v_n\sigma_n \geq 1$ always holds. By considering $(\sfrac{\tau_k}{\sigma_k})^r$ we can analogously prove \ref{it:seq:exp:alm_decr}$\Rightarrow$\ref{it:seq:exp:tail}.

\ref{it:seq:exp:partial}$\Rightarrow$\ref{it:seq:exp:shifted} Let $c>0$ be a constant with $v_n\sigma_n\leq c$ for all $n\geq 1$. Because of the monotonicity of $\sigma$ we get for $k,n_0\geq 1$
\[ 
c^s
\geq v_{k+n_0}^s\sigma_{k+n_0}^s
= \sum_{i=1}^{k+n_0} \Bigl(\frac{\sigma_{k+n_0}}{\sigma_i}\Bigr)^s 
\geq \sum_{i=k}^{k+n_0} \Bigl(\frac{\sigma_{k+n_0}}{\sigma_i}\Bigr)^s
\geq \Bigl(\frac{\sigma_{k+n_0}}{\sigma_k}\Bigr)^s (n_0+1)\eqspace.
\]
Choosing $n_0 \coloneqq \lceil c^s \rceil$ yields for $k\geq 1$ 
\[ 
\frac{\sigma_{k+n_0}}{\sigma_k} \leq \frac{c}{(n_0+1)^{\sfrac{1}{s}}} \leq \frac{c}{(c^s +1)^{\sfrac{1}{s}}} < 1\eqspace.
\]

\ref{it:seq:exp:tail}$\Rightarrow$\ref{it:seq:exp:shifted} Let $c>0$ be a constant with $\tau_k\leq c \sigma_k$ for all $k\geq 1$. Because of the monotonicity of $\sigma$ we get for $k,n_0\geq 1$
\[
c^r 
\geq \frac{\tau_k^r}{\sigma_k^r}
= \sum_{n=k}^\infty \Bigl(\frac{\sigma_n}{\sigma_k}\Bigr)^r 
\geq \sum_{n=k}^{k+n_0} \Bigl(\frac{\sigma_n}{\sigma_k}\Bigr)^r 
\geq \Bigl(\frac{\sigma_{k+n_0}}{\sigma_k}\Bigr)^r (n_0+1)\eqspace.
\]
Hence Statement~\ref{it:seq:exp:shifted} follows along the same line as \ref{it:seq:exp:partial}$\Rightarrow$\ref{it:seq:exp:shifted}.

\ref{it:seq:exp:shifted}$\Rightarrow$\ref{it:seq:exp:alm_decr} For $k\leq n$ there is a unique $m\geq 0$ with $k + mn_0 \leq n < k+(m+1)n_0$. Using the monotonicity of $\sigma$ and Assumption~\ref{it:seq:exp:shifted} $m$-times we get
\[ 
\sigma_n \leq \sigma_{k+mn_0} \leq \sigma_k a^m \leq \frac{\sigma_k}{a} a^{\frac{n-k}{n_0}}= \frac{\sigma_k}{a} b^{k-n}
\]
with $b=a^{-\sfrac{1}{n_0}}>1$. Hence the supremum is bounded by $a^{-1}$.
\end{proof}

\begin{lem}[Doubling Condition]\label{lem:seq:doubling}
Let $\sigma=(\sigma_k)_{k\geq 1}$ with $\sigma_k>0$ and $\sigma_k\searrow 0$. Then the following statements are equivalent:
\begin{enumerate}
\item\label{it:seq:doubling:doubling} $\sigma_n\asymp\sigma_{2n}$.
\item\label{it:seq:doubling:o_varying} For all $\lambda>0$ the function $f(x) \coloneqq \sigma_{\lfloor x\rfloor + 1}$ satisfies $f(x)\asymp f(\lambda x)$ for $x>0$.
\item\label{it:seq:doubling:alm_incr} $\inf_{k\leq n}\frac{\sigma_n n^\alpha}{\sigma_k k^\alpha} > 0$ for some $\alpha>0$.
\item\label{it:seq:doubling:geo_mean} $\sigma_n\asymp (\sigma_1\cdot\ldots\cdot\sigma_n)^{\sfrac{1}{n}}$.
\end{enumerate}
\end{lem}

Note that the symbol $\asymp$ in Statement~\ref{it:seq:doubling:o_varying} means that for all $\lambda>0$ there are constants $c_1,c_2>0$, depending on $\lambda>0$, with $c_1f(x)\leq f(\lambda x)\leq c_2f(x)$ for all $x>0$. Moreover, Statement~\ref{it:seq:doubling:alm_incr} implies $\sigma_n\succcurlyeq n^{-\alpha}$ and hence $\sigma$ decreases at most polynomially. 

\begin{proof}
\ref{it:seq:doubling:doubling}$\Leftrightarrow$\ref{it:seq:doubling:alm_incr} This has already been pointed out by \citet[p.~482]{K2005} and is a direct consequence of the monotonicity of $\sigma$. 

\ref{it:seq:doubling:doubling}$\Leftrightarrow$\ref{it:seq:doubling:o_varying} Statement \ref{it:seq:doubling:o_varying}, for $\lambda=2$ and $x=n-\sfrac{1}{2}$, directly implies \ref{it:seq:doubling:doubling}.
For the inverse implication we first show that 
\begin{equation}\label{eq:seq:doubling:o_varying:temp}
\lfloor nx\rfloor + 1 \leq n(\lfloor x\rfloor + 1)
\end{equation}
holds for all $n\geq 1$ and all $x>0$. 
To this end, let $0\leq r<1$ with $x = \lfloor x\rfloor + r$. Since the strict inequality $nx = n\lfloor x\rfloor + nr < n\lfloor x\rfloor + n$ holds and the right hand side is an integer we find $\lfloor nx\rfloor \leq n\lfloor x\rfloor + n-1$ which is equivalent to \eqref{eq:seq:doubling:o_varying:temp}.
Now, to the implication \ref{it:seq:doubling:doubling}$\Rightarrow$\ref{it:seq:doubling:o_varying}. Let $c>0$ be the doubling constant of $\sigma$, i.e.\ $\sigma_{2n}\geq c\sigma_n$ for all $n\geq 1$. Using the monotonicity of $\sigma$, the inequality in \eqref{eq:seq:doubling:o_varying:temp}, and \ref{it:seq:doubling:doubling} we find
\begin{equation}\label{eq:seq:doubling:o_varying:doubling}
f(2x) 
= \sigma_{\lfloor 2x\rfloor + 1} 
\geq \sigma_{2(\lfloor x\rfloor + 1)}
\geq c\sigma_{\lfloor x\rfloor + 1} 
= c f(x)
\eqspace.
\end{equation}
Finally, for fixed $\lambda\geq 1$ we choose an $m\geq 1$ with $2^m\geq\lambda$. The monotonicity of $f$ and an $m$-times application of \eqref{eq:seq:doubling:o_varying:doubling} yields \ref{it:seq:doubling:o_varying}. The case $0<\lambda<1$ can be easily deduced from the case $\lambda>1$. 

\ref{it:seq:doubling:alm_incr}$\Rightarrow$\ref{it:seq:doubling:geo_mean} Because of the monotonicity of $\sigma$ we always have $(\sigma_1\cdot\ldots\cdot\sigma_n)^{\sfrac{1}{n}} \geq \sigma_n$. For $c \coloneqq \inf_{k\leq n}\frac{\sigma_n n^\alpha}{\sigma_k k^\alpha} > 0$ we have $\sigma_k\leq c^{-1} \sigma_n n^\alpha k^{-\alpha}$ for all $k\leq n$. Since Stirling's formula yields $(n!)^{\sfrac{1}{n}}\asymp n$ we get
\[ 
(\sigma_1\cdot\ldots\cdot\sigma_n)^{\sfrac{1}{n}} 
\leq c^{-1}\sigma_n\frac{n^\alpha}{(n!)^{\sfrac{\alpha}{n}}}
\asymp \sigma_n\eqspace.
\]

\ref{it:seq:doubling:geo_mean}$\Rightarrow$\ref{it:seq:doubling:doubling} Let $c>0$ with $\sigma_n \leq (\sigma_1\cdot\ldots\cdot\sigma_n)^{\sfrac{1}{n}} \leq c \sigma_n$ for all $n\geq 1$. Then
\[ 
c\sigma_{2n}
\geq  (\sigma_1\cdot\ldots\cdot\sigma_{2n})^{\frac{1}{2n}}
= (\sigma_1\cdot\ldots\cdot\sigma_n)^{\frac{1}{2n}} (\sigma_{n+1}\cdot\ldots\cdot\sigma_{2n})^{\frac{1}{2n}}
\geq \sqrt{\sigma_n \sigma_{2n}}\eqspace.
\]
is satisfied for all $n\geq 1$. Hence we have $c^2\sigma_{2n}\geq \sigma_n \geq \sigma_{2n}$ for all $n\geq 1$.
\end{proof}

\begin{lem}[Tail Sequence]\label{lem:seq:tail}
Let $r>0$, $\sigma=(\sigma_k)_{k\geq 1}$ with $\sigma_k>0$ and $\sigma_k\searrow 0$ and $\tau$ be the tail sequence given by \eqref{eq:tail_def}. Then the following statements hold:
\begin{enumerate}
\item\label{it:seq:tail:alp} The following statements are equivalent:
\begin{enumerate}
\item\label{it:seq:tail:alp:alm_decr} $\sup_{k\leq n} \frac{\sigma_n n^\alpha}{\sigma_k k^\alpha}<\infty$ for some $\alpha >\sfrac{1}{r}$.
\item\label{it:seq:tail:alp:inf_frac} Condition~(\ALP): $\tau_n\preccurlyeq \sigma_n n^{\sfrac{1}{r}}$.
\setcounter{temp}{\value{enumii}}
\end{enumerate}
\item\label{it:seq:tail:amp} The following statements are equivalent:
\begin{enumerate}
\setcounter{enumii}{\value{temp}}
\item\label{it:seq:tail:amp:doubling} $\tau_n\asymp \tau_{2n}$.
\item\label{it:seq:tail:amp:sup_frac} Condition~(\AMP): $\tau_n\succcurlyeq \sigma_n n^{\sfrac{1}{r}}$.
\end{enumerate}
\item\label{it:seq:tail:doubling} Condition $\sigma_n\asymp\sigma_{2n}$ implies $\tau_n\asymp \tau_{2n}$. If, in addition, \ref{it:seq:tail:alp:alm_decr} is satisfied, then $\sigma_n\asymp\sigma_{2n}$ is satisfied if and only if $\tau_n\asymp \tau_{2n}$ is satisfied.
\end{enumerate}
\end{lem}

\begin{proof}
\ref{it:seq:tail:alp:alm_decr}$\Rightarrow$\ref{it:seq:tail:alp:inf_frac} For $c \coloneqq \sup_{k\leq n} \frac{\sigma_n n^\alpha}{\sigma_k k^\alpha}<\infty$ we get
\[ 
\frac{\tau_k^r}{k\sigma_k^r} 
= \frac{1}{k}\sum_{n=k}^\infty \Bigl(\frac{\sigma_n}{\sigma_k}\Bigr)^r
\leq c^r k^{\alpha r - 1}\sum_{n=k}^\infty n^{-\alpha r} 
\]
for all $k\geq 1$. Estimating the remaining sum using integrals we get the assertion
\[
k^{\alpha r - 1}\sum_{n=k}^\infty n^{-\alpha r} 
\leq k^{\alpha r - 1}\biggl(k^{-\alpha r} + \int_k^\infty t^{-\alpha r}\ \d t\biggr)
\leq \frac{\alpha r}{\alpha r -1}\eqspace.
\]
\ref{it:seq:tail:alp:inf_frac}$\Rightarrow$\ref{it:seq:tail:alp:alm_decr} is a consequence of \citet[Theorem~2.6.3]{BiGoTe1989} to the positive and measurable function $f(x) \coloneqq x\sigma_{\lfloor x\rfloor}^r$ for $x\geq 1$. To this end, we recall the definition of \emph{almost decreasing} functions from \cite[Section~2.2.1]{BiGoTe1989} and the \emph{Matuszewska index} $\alpha(f)$ of $f$, defined in \cite[Section~2.1.2]{BiGoTe1989}. Moreover, we have 
\[ 
\alpha(f)=\inf\bigl\{\alpha\in\R:\ x^{-\alpha} f(x) \text{ is almost decreasing}\bigr\}
\]
according to \cite[Theorem~2.2.2]{BiGoTe1989}. Since $x^{-1}f(x)$ is decreasing we have $\alpha(f)\leq 1 <\infty$ and hence $f$ is of \emph{bounded increase}, i.e.\ $f\in\text{BI}$, see \cite[p.~71]{BiGoTe1989} for a definition. Consequently, \cite[Theorem~2.6.3 (d)]{BiGoTe1989} is applicable to the function $f$. For $\tilde{f}(x) \coloneqq \int_x^\infty \sfrac{f(t)}{t}\,\d t$ we have 
\[ 
\frac{f(x)}{\tilde{f}(x)} 
= \frac{x\sigma_{\lfloor x\rfloor}^r}{\tau_{\lfloor x\rfloor}^r - (x - \lfloor x\rfloor)\sigma_{\lfloor x\rfloor}^r} 
\geq \frac{x \sigma_{\lfloor x\rfloor}^r}{\tau_{\lfloor x\rfloor}^r}
\geq \frac{\lfloor x\rfloor \sigma_{\lfloor x\rfloor}^r}{\tau_{\lfloor x\rfloor}^r}
\geq c^{-r}
\]
for all $x\geq 1$, where $c>0$ is a constant satisfying $\tau_n \leq c \sigma_n n^{\sfrac{1}{r}}$ for all $n\geq 1$. Therefore, $\liminf_{x\to\infty}\sfrac{f(x)}{\tilde{f}(x)}>0$ and \cite[Theorem~2.6.3 (d)]{BiGoTe1989} yields $\alpha(f)<0$. Consequently, there is a $\alpha_0 < 0$ such that $x^{-\alpha_0}f(x)$ is almost decreasing. The definition of almost decreasing gives us the assertion with $\alpha = \frac{1-\alpha_0}{r} > \sfrac{1}{r}$.

\ref{it:seq:tail:amp:doubling}$\Rightarrow$\ref{it:seq:tail:amp:sup_frac} This is from \cite[first equation on p.~45]{K2008}. \ref{it:seq:tail:amp:sup_frac}$\Rightarrow$\ref{it:seq:tail:amp:doubling} The following idea is from \cite[proof of Theorem~4]{BoSe1973}. According to our assumption the sequence
\[ 
\rho_n \coloneqq n\Bigl(1 - \frac{\tau_{n+1}^r}{\tau_n^r}\Bigr) 
= n \frac{\tau_n^r - \tau_{n+1}^r}{\tau_n^r} 
= \frac{n \sigma_n^r}{\tau_n^r} 
\]
is positive and bounded. Building a telescope product we get
\[ 
\frac{\tau_{n}^r}{\tau_1^r} =\prod_{k=1}^{n-1} \frac{\tau_{k+1}^r}{\tau_k^r} = \prod_{k=1}^{n-1} \Bigl(1 - \frac{\rho_k}{k}\Bigr)\eqspace.
\]
Since $0<1-\frac{\rho_k}{k}<1$ this gives $\tau_n^r = \exp\circ\log(\tau_n^r) = \exp(\gamma_n - \sum_{k=1}^{n-1}\sfrac{\rho_k}{k})$ with
\[
\gamma_n \coloneqq \log \tau_1^r + \sum_{k=1}^{n-1} \biggl(\log\Bigl(1-\frac{\rho_k}{k}\Bigr) + \frac{\rho_k}{k}\biggr)\eqspace.
\]
Below we will prove that $(\gamma_n)_{n\geq 1}$ converges and hence the assertion is a consequence of this representation of $\tau_n^r$ according to \cite[Theorem~2]{DjTo2004}. Now, to the convergence of $(\gamma_n)_{n\geq 1}$. Since $(\rho_k)_{k\geq 1}$ is bounded the sequence $a_k \coloneqq \sfrac{\rho_k}{k}$ is square summable. Without loss of generality we assume that there is a $0<q<1$ with $a_n<q$ for all $n\geq 1$. Using the Taylor series of the logarithm we get
\[ 
\log(1-a_k) + a_k = -\sum_{\ell=1}^\infty \frac{a_k^\ell}{\ell} + a_k = -\sum_{\ell=2}^\infty \frac{a_k^\ell}{\ell}\eqspace.
\]
Additionally, for $\ell\geq 2$ we have the estimate $\sum_{k=1}^\infty a_k^\ell \leq \|a\|_{\ell_2}^2 q^{\ell-2}$.
Together we get the absolute convergence of the series
\[ 
\sum_{k=1}^\infty|\log(1-a_k) + a_k| = \sum_{k=1}^\infty \sum_{\ell=2}^\infty\frac{a_k^\ell}{\ell} = \sum_{\ell=2}^\infty\frac{1}{\ell}\sum_{k=1}^\infty a_k^\ell \leq \frac{\|a\|_{\ell_2}^2}{q^2} \sum_{\ell=2}^\infty \frac{q^\ell}{\ell}<\infty\eqspace.
\]

\ref{it:seq:tail:doubling} According to our assumption there is a constant $c>0$ with $\sigma_{2n}\geq c\sigma_n$ for all $n\geq 1$. Then the assertion follows by
\[ 
\tau_{2n}^r \geq \sum_{k=n}^\infty \sigma_{2k}^r \geq c^r \sum_{k=n}^\infty \sigma_k^r = c^r \tau_n^r\eqspace.
\]
For the inverse we additionally assume \ref{it:seq:tail:alp:alm_decr} and hence we have also \ref{it:seq:tail:alp:inf_frac} and \ref{it:seq:tail:amp:sup_frac}, i.e.\ $\tau_n\asymp \sigma_n n^{\sfrac{1}{r}}$. Consequently, $\sigma_{2n}\asymp \tau_{2n} (2n)^{-\sfrac{1}{r}} \asymp \tau_n n^{-\sfrac{1}{r}}\asymp \sigma_n$ is satisfied.
\end{proof}

\begin{singlespace}

\bibliographystyle{../_style/bibliographystyle}
\bibliography{../_style/literatur}

\end{singlespace}


\end{document}